\documentclass[12pt]{amsart}
\usepackage{epsfig,color}
\usepackage{blindtext}
\usepackage{graphicx}
\usepackage{enumitem}
\usepackage{url}
\usepackage{amssymb}
\usepackage{graphicx,import}
\usepackage{comment}
\usepackage{esint}
\usepackage{xcolor}
\usepackage{mathtools}
\usepackage{comment}
\usepackage{tikz}

\allowdisplaybreaks

\usepackage[margin = 1in] {geometry}

\usepackage{hyperref}
\usepackage{dsfont}

\newtheorem{theorem}{Theorem}[section]

\newtheorem{proposition}[theorem]{Proposition}
\newtheorem{lemma}[theorem]{Lemma}
\newtheorem{corollary}[theorem]{Corollary}

\theoremstyle{definition}

\newtheorem*{remark*}{Remark}

\newtheorem{remark}[theorem]{Remark}

\theoremstyle{claim}
\newtheorem*{claim}{Claim}

\numberwithin{equation}{section}

\newcommand{\cA}{\mathcal{A}}
\newcommand{\cF}{\mathcal{F}}
\newcommand{\cH}{\mathcal{H}}
\newcommand{\cI}{\mathcal{I}}
\newcommand{\R}{\mathbf{R}}

\newcommand{\eps}{\varepsilon}
\newcommand{\g}{{g_N}}

\DeclareMathOperator{\Hess}{Hess}
\DeclareMathOperator{\Ric}{Ric}
\DeclareMathOperator{\BiRic}{BiRic}
\DeclareMathOperator{\Tr}{Tr}
\DeclareMathOperator{\Rs}{R}

\makeatletter
\def\@defaultbiblabelstyle#1{[#1]}
\makeatother

\title{Stable minimal hypersurfaces in $\mathbf{R}^5$}

\author{Otis Chodosh}
\address{Department of Mathematics, Stanford University, Building 380, Stanford, CA 94305, USA}
\email{ochodosh@stanford.edu}

\author{Chao Li}
\address{Courant Institute, New York University, 251 Mercer St, New York, NY 10012, USA}
\email{chaoli@nyu.edu}

\author{Paul Minter}
\address{Department of Mathematics, Stanford University, Building 380, Stanford, CA 94305, USA}
\email{pminter@stanford.edu}

\author{Douglas Stryker}
\address{Department of Mathematics, Stanford University, Building 380, Stanford, CA 94305, USA}
\email{dstryker@stanford.edu}

\begin{document}

\maketitle

\begin{abstract}
    We show that a complete, two-sided stable minimal hypersurface in $\mathbf{R}^5$ is flat.
\end{abstract}

\section{Introduction}
A two-sided immersion $M^n\to\R^{n+1}$ is \emph{minimal} if its mean curvature vector vanishes. A minimal immersion is \emph{stable} if 
\[
\int_M |A_M|^2 \varphi^2 \leq \int_M |\nabla \varphi|^2
\]
for all $\varphi \in C^\infty_c(M)$, where $A_M$ is the second fundamental form of the immersion. The stable Bernstein problem asks whether a complete, connected, two-sided stable minimal hypersurface in $\R^{n+1}$ must be an affine hyperplane. We resolve here the stable Bernstein problem in $\R^5$:

\begin{theorem}\label{thm:main}
A complete, connected, two-sided stable minimal immersion $M^4 \to \R^5$ is an affine hyperplane.
\end{theorem}

The stable Bernstein problem was resolved in $\R^3$ by do Carmo--Peng, Fischer-Colbrie--Schoen, and Pogorelov  \cite{dCP:R3,FCS:stable,P} and recently in $\R^4$ by the first- and second-named authors  \cite{CL:R4} (subsequently, two alternative proofs in $\R^4$ were found in \cite{CL:aniso} and \cite{CMR}). After the first version of this paper appeared, Mazet was able to refine our methods to resolve the stable Bernstein problem in $\R^6$ in the affirmative \cite{Mazet} (see Remark \ref{rema:further-dev}). The stable Bernstein problem remains open in $\R^7$ but holds assuming $M$ has (extrinsic) Euclidean volume growth by work of Schoen--Simon and Simons \cite{SS,Simons:cone} (for embeddings) and Bellettini \cite{B} (for immersions).  On the other hand, non-flat stable (area-minimizing) minimal immersions in $\R^{8}$ (and beyond) were found by Bombieri--de Giorgi--Giusti \cite{BDG}.

It is well-known that the validity of the stable Bernstein property is equivalent to an \emph{a priori} interior curvature bound for stable minimal hypersurfaces. We recall that a two-sided minimal immersion into a Riemannian manifold $M^n\to(X^{n+1},g)$ is stable if 
\[
\int_M (|A_M|^2+\Ric_g(\nu,\nu)) \varphi^2 \leq \int_M |\nabla \varphi|^2
\]
for all $\varphi \in C^\infty_c(M)$ (we additionally require that $\varphi|_{\partial M} = 0$ if $\partial M \neq \varnothing$), where $A_M$ is the second fundamental form of $M$ and $\Ric_g(\nu,\nu)$ is the Ricci curvature of the ambient metric $g$ in the normal direction. As in \cite[Corollary 2.5]{CLS:stable}, Theorem \ref{thm:main} implies interior curvature estimates for stable minimal immersions in $5$-manifolds that only depend on a norm for the ambient sectional curvature. 
\begin{corollary}\label{coro:curv-est}
Let $(X^5,g)$ be a complete Riemannian manifold with bounded sectional curvature $\lvert\sec_g\rvert\leq K$. Then any compact, two-sided stable minimal immersion $M^4\to (X^5,g)$ satisfies
\[
|A_M|(q) \min\{1,d_M(q,\partial M)\} \leq C(K)
\]
for all $q\in M$. 
\end{corollary}

To prove Theorem \ref{thm:main}, we show that a complete, two-sided, stable minimal immersion $M^4\to\R^5$ has intrinsic Euclidean volume growth. In fact, as in \cite{CL:aniso}, this can be localized in the spirit of Pogorelov's area bounds \cite{P} for stable minimal surfaces in $\R^3$. 

\begin{theorem}\label{thm:localize}
    Let $F: M^4 \to \R^5$ be a simply connected, two-sided stable minimal immersion so that $F(x_0) = 0 \in \R^5$ for some $x_0\in M$, $\partial M$ is connected, and $F: M \to B_{\R^5}(0, 1)$ is proper. Then
    \[ \cH^4(M_{\rho_0}^*) \leq 8\pi^2, \]
    where $M_{\rho_0}^*$ is the connected component of $F^{-1}(B_{\R^5}(0,\rho_0))$ containing $x_0$ and $\rho_0 = e^{-11\pi}$.
\end{theorem}

Similarly, we can give a geometric characterization of minimal hypersurfaces in $\R^5$ with finite Morse index, generalizing the well-known results of Gulliver, Fischer-Colbrie, and Osserman in $\R^3$ \cite{Gulliver,FC,Osserman}. See \cite{CL:R4} for the corresponding result in $\R^4$. Recall that a two-sided minimal immersion $M^4\to\R^5$ has \emph{finite Morse index} if
\[
\sup\{\dim V : V \subset C^\infty_c(M) \textrm{ is a subspace with } Q(f,f) < 0 \textrm{ for all } 0 \neq f \in V\} < \infty
\]
where $Q(f,f) = \int_M |\nabla f|^2 - \int_M|A_M|^2 f^2$. We additionally recall \cite[Section 2]{S:sym} that an end $E$ of a minimal immersion $M^4\to \R^5$ is \emph{regular at infinity} if it is contained in the graph of a function $w$ on a hyperplane $\Pi$ with asymptotics 
\[
w(x) = b + a|x|^{-2} +\sum_{j=1}^4 c_jx_j|x|^{-4} + O(|x|^{-4})
\]
where $a,b,c_1,\dots,c_4$ are constants and $x_1,\dots,x_4$ are Euclidean coordinates on $\Pi$. 

\begin{theorem}\label{theo:morse}
    A complete, two-sided minimal immersion $M^4\to\R^5$ has finite Morse index if and only if it has finite total curvature, i.e., $\int_M |A_M|^4 < \infty$, in which case $M$ is properly immersed, has finitely many ends, each end of $M$ is regular at infinity, and $M$ has Euclidean volume growth, namely $|M\cap B_{\mathbf{R}^5}(0,\rho)|\leq C\rho^4$ for all $\rho>0$; here $C = C(M)$. 
\end{theorem}

Previously, Tysk \cite{T} proved the same result (relying on \cite{SSY} for $n \leq 5$ and \cite{SS,B} for $n = 6$) for a complete, two-sided minimal immersion $M^n\to \R^{n+1}$ ($3\le n\le 6$) under the additional assumption of extrinsic Euclidean volume growth, i.e.~$\sup_{\rho>0} \frac{\mathrm{Vol}(M\cap B_{\R^{n+1}}(0,\rho))}{\rho^n}<\infty$. In light of Theorem \ref{theo:morse}, it would be interesting to further investigate the relation between the topology (e.g.~the number of ends) and the Morse index of finite index minimal immersions (cf.~\cite{Li2017}), and classify minimal hypersurfaces in $\R^5$ with low index (cf. \cite{ChodoshMaximoI,ChodoshMaximoII}). Another natural question here is to construct nontrivial minimal immersions with finite Morse index in $\R^{5}$ (cf.~\cite{Coutant}). We anticipate that Theorem \ref{theo:morse} and these problems should also be relevant to the study of stable/finite index minimal hypersurfaces in $5$-manifolds in a manner similar to \cite{CLS:stable}.

\subsection{Discussion of results and methods}
Let $M^n \to \R^{n+1}$ be a complete, two-sided stable minimal immersion. The main difficulty in resolving the stable Bernstein problem is that the extrinsic and intrinsic geometry of $M^n \to \R^{n+1}$ could be \emph{a priori} very complicated. For example, if $g$ is the induced metric on $M$, the manifold $(M,g)$ might have exponential volume growth. Furthermore, stability of the immersion does not directly imply any pointwise curvature condition on $g$, while minimality only implies that $g$ has \emph{non-positive} Ricci curvature via the Gauss equation. 

The strategy used in this article is motivated by the one developed by the first- and second-named authors for $M^3\to\R^4$ in \cite{CL:aniso}. More precisely, we let $g$ denote the induced metric on $M$ and consider the conformally changed metric $\tilde{g} = r^{-2}g$, where $r$ is the Euclidean distance function from the origin. This conformal change was first considered by Gulliver--Lawson in their study of isolated singularities in stable minimal hypersurfaces \cite{GL:conf}. Note that if $M$ is a hyperplane containing the origin, then  $(M,\tilde g)$ will be the standard round cylinder  $\R \times \mathbf{S}^{n-1}$.

A key insight of Gulliver--Lawson is that $\tilde{g}$ has uniformly positive scalar curvature in a weak spectral sense; namely, we have
\begin{equation}\label{eqn:schrodinger}
    -\tilde{\Delta} + \frac{1}{2}\left(\tilde{R} - \frac{n(n-2)}{2}\right) \geq 0,
\end{equation}
where $\tilde{R}$ is the scalar curvature and $\tilde{\Delta}$ is the (nonpositive) Laplace--Beltrami operator of $\tilde{g}$. It is now well-known that a 3-manifold $N^3$ with uniformly positive scalar curvature has \emph{macroscopic dimension one} in several ways:
\begin{itemize}
    \item \emph{Distance-sense}: $N$ has bounded Urysohn 1-width, see \cite{Katz} and \cite{LM:waist}.
    \item \emph{Area-sense}: $N$ admits evenly spaced separating surfaces with uniformly bounded area, see \cite{CL:soapbubbles} and \cite{CL:aniso}.
    \item \emph{Volume-sense}: If $N$ has nonnegative Ricci curvature, then $N$ has linear volume growth, see \cite{MW:vol} and \cite{CLS:vol}.
\end{itemize}
Since each of these results can be proved using the $\mu$-bubble localization method introduced by Gromov\footnote{The stability of $\mu$-bubbles, which are hypersurfaces with prescribed mean curvature, have been studied by many authors. In this paper, we make essential use of the fact that they can be localized, as observed by Gromov \cite{Gromov:mububbles}.} \cite{Gromov:mububbles} (see also \cite{Gromov:metric-inequalities}), and since the weaker condition \eqref{eqn:schrodinger} suffices to carry out the $\mu$-bubble argument, these observations hold for 3-manifolds satisfying \eqref{eqn:schrodinger}. In particular, an appropriate version of the fact that $(M,\tilde g)$ has \emph{macroscopic dimension one} was used in \cite{CL:aniso} to deduce the stable Bernstein theorem for $M^3\to \R^4$. 

In higher dimensions, the positive scalar curvature property \eqref{eqn:schrodinger} of $(M,\tilde g)$ appears to be too weak to deduce the stable Bernstein property. (In particular, $\R^{n-2}\times \mathbf{S}^2$ has uniformly positive scalar curvature.) The first main idea in the proof of Theorem \ref{thm:main} is to replace scalar curvature with a stronger curvature condition, namely \emph{bi-Ricci curvature}. The bi-Ricci curvature of two orthonormal vectors $v, w \in T_pM$ is defined as
\[ \BiRic(v,w) = \Ric(v,v) + \Ric(w,w) - \Rs(v,w,w,v) \]
where $\Rs$ is the curvature tensor. Alternatively, $\BiRic(v,w)$ is the sum of sectional curvatures of 2-planes intersecting the plane spanned by $v$ and $w$. We note that the bi-Ricci curvature of a 3-manifold is a multiple of the scalar curvature. Importantly, $\R \times \mathbf{S}^{n-1}$ has uniformly positive bi-Ricci curvature, while $\R^k\times \mathbf{S}^{n-k}$ does not for any $k > 1$. 

\begin{remark}
The bi-Ricci curvature was introduced by Shen--Ye \cite{SY:stable,SY:general} motivated by its the relationship with stable minimal hypersurfaces. More recently, Brendle--Hirsch--Johne studied  bi-Ricci curvature as part of a general notion of curvatures that interpolate between Ricci and scalar curvature \cite{BHJ}.
\end{remark}

In Theorem \ref{thm:spectralbiRic}, we prove that for $M^4\to \R^5$ a complete, two-sided stable minimal immersion, the Gulliver--Lawson conformal metric $(M,\tilde g)$ has uniformly positive bi-Ricci curvature in the weak spectral sense; namely, we have
\begin{equation}\label{eqn:schrodiner2}
    -\tilde{\Delta} + (\tilde{\lambda}_{\BiRic} - 1) \geq 0,
\end{equation}
where $\tilde{\lambda}_{\BiRic}(x)$ is the smallest bi-Ricci curvature of $\tilde{g}$ at $x$. This computation follows the general strategy introduced by Gulliver--Lawson \cite{GL:conf}, but is considerably more involved. At this point, we must leverage the improved positivity \eqref{eqn:schrodiner2} to prove ``one-dimensionality'' of $(M^4,\tilde g)$, and then use this to conclude the stable Bernstein theorem.

In \cite{xu:biRic}, Xu showed that if $N^n$ has uniformly positive bi-Ricci curvature and $n \leq 5$, then $N$ admits $\mu$-bubbles with uniformly positive Ricci curvature in the weak spectral sense. We generalize his arguments (for $n = 4$) to \emph{spectral} uniformly positive bi-Ricci curvature. This construction produces an exhaustion by (warped) $\mu$-bubbles $\Sigma$ that satisfy
\begin{equation}\label{eqn:schrodiner3}
    -\Delta^{\Sigma} + \frac{3}{4}\left(\lambda_{\Ric}^{\Sigma} - \frac{1}{2}\right) \geq 0,
\end{equation}
where $\lambda_{\Ric}^{\Sigma}(x)$ is the smallest eigenvalue of the Ricci curvature of $\Sigma$ at $x$. The spectral condition \eqref{eqn:schrodiner3} can be thought of as a weak form of uniform positivity of the Ricci curvature of the $\mu$-bubble $\Sigma$. Indeed, \eqref{eqn:schrodiner3} implies that $\Sigma$ has uniformly bounded diameter by a result of Shen--Ye \cite{SY:general} (see Theorem \ref{thm:diameter_and_vol}). In particular, this implies that $(M,\tilde g)$ has bounded Urysohn $1$-width. 

The fact that $(M,\tilde g)$ has bounded Urysohn $1$-width does confirm that a complete, two-sided stable minimal immersion $M^4\to\R^5$ has controlled geometry in a certain sense, but does not seem sufficient to conclude that $M$ is an affine hyperplane. Indeed, in the proof of the stable Bernstein theorem for $M^3\to\R^4$ in \cite{CL:aniso}, the key property of the $\mu$-bubbles is that they have bounded area (in addition to bounded diameter). In this case, when the $\mu$-bubbles are two-dimensional, the area bound follows directly from a spectral condition like \eqref{eqn:schrodiner3} thanks to the Gauss--Bonnet theorem. 

In higher dimensions (as considered in this paper) the Bishop--Gromov volume comparison gives uniform volume upper bounds from Ricci curvature lower bounds. We consider a weighted version of Bray's proof \cite{Bray:thesis} of the Bishop theorem via the isoperimetric profile and use this to deduce a sharp volume comparison theorem for $3$-dimensional manifolds having spectrally positive Ricci curvature (see Theorem \ref{thm:diameter_and_vol}). 

\begin{remark}
A different spectral version of the Bishop--Gromov theorem was obtained by Carron for manifolds with a Euclidean Sobolev inequality and a (strong) form of non-negative Ricci curvature in a spectral sense \cite{Carron:Kato}. Note that this result does not seem applicable to our setting since the $\mu$-bubbles are compact and thus cannot admit a Euclidean Sobolev inequality. 
\end{remark}

Granted these two ingredients, i.e., spectral positivity of bi-Ricci curvature for the Gulliver--Lawson conformal metric and volume bounds for the $3$-dimensional $\mu$-bubbles, we can follow the  arguments of \cite{CL:aniso} to deduce the stable Bernstein theorem for $n=4$.

\begin{remark}\label{rema:further-dev}
    There have been several exciting developments in this area that occurred after this paper was first posted. Antonelli--Xu have generalized the spectral Bishop--Gromov result to hold in all dimensions \cite{AntonelliXu:spect}. Mazet subsequently combined their spectral Bishop--Gromov result with a delicate refinement of the strategy used here to resolve the stable Bernstein problem in $\mathbf{R}^6$ \cite{Mazet}. (See also \cite{AX:note,Tam}.)
\end{remark}

\subsection{Organization}
We review the notation and conventions in Section \ref{sec:not}. Then, we compute the spectral curvature properties of the Gulliver--Lawson conformal metric in Section \ref{sec:GL}. We discuss $\mu$-bubble existence and stability in Section \ref{sec:mu} and then use the stability inequality to prove geometric estimates for the $\mu$-bubbles in Section \ref{sec:mu-geo}. Finally, in Section \ref{sec:proofs} we prove the results stated in the introduction. 

\subsection{Acknowledgements}
O.C. was supported by a Terman Fellowship and an NSF grant
(DMS-2304432). C.L. was supported by an NSF grant (DMS-2202343) and a Simons Junior Faculty Fellowship. This research was conducted during the period P.M. served as a Clay Research Fellow. We are grateful to Laurent Mazet for pointing out an error in the proof of Theorem \ref{thm:secondvar} in an earlier version of the paper as well as to Michael Eichmair, Thomas K\"orber, and the referees for their careful reading of this manuscript and several helpful suggestions.

\section{Notation and Conventions}\label{sec:not}
Let $(M^n, g)$ be a Riemannian manifold. At $p \in M$, let $\{e_i\}_{i=1}^n$ be an orthonormal basis for $T_pM$. We write $\Rs(\,\cdot\,,\,\cdot\,,\,\cdot\,,\,\cdot\,)$ for the curvature operator with the convention that $\Rs(e_i,e_j,e_j,e_i)$ is the sectional curvature of the $2$-plane spanned by $e_i,e_j$. We define at $p$:
\begin{itemize}
    \item \emph{the Ricci curvature tensor}: \[\Ric(e_1, e_1) = \sum_{i=2}^n \Rs(e_1, e_i, e_i, e_1).\]
    \item \emph{the minimum Ricci curvature scalar}: \[\lambda_{\Ric} = \inf_{|v| = 1} \Ric(v, v).\]
    \item \emph{the bi-Ricci curvature}:
    \[ \BiRic(e_1, e_2) = \sum_{i=2}^n \Rs(e_1, e_i, e_i, e_1) + \sum_{j=3}^n \Rs(e_2, e_j, e_j, e_2). \]
    \item \emph{the minimum bi-Ricci curvature scalar}: \[\lambda_{\BiRic} = \inf_{\{v, w\}\ \text{orthonormal}} \BiRic(v, w).\]
\end{itemize}

Let $M^n \to (X^{n+1},g_X)$ be a smooth immersion of codimension one that is two-sided. Let $\nu$ be a smooth unit normal vector field along $M$. We define:
\begin{itemize}
    \item \emph{the second fundamental form}: \[A(X, Y) = -\langle \nabla^N_XY, \nu \rangle.\]
    \item \emph{the mean curvature}: \[H = \sum_{i=1}^n A(e_i, e_i).\]
\end{itemize}
In particular, the unit sphere in $\R^{n+1}$ has positive mean curvature with respect to the outward unit normal. Given our  conventions, for a smooth family of such immersions $\{F_t : M^n \to (X^{n+1}, g_X)\}_{t}$, we have
\[ \frac{d}{dt}\cH^n_{F_t^*g_X}(M) = \int_{M} H_{F_t}\left\langle \nu_{F_t}, \frac{dF_t}{dt}\right\rangle\ d\cH^n_{F_t^*g}, \]
where $F_t^*g_X$ is the pullback metric under $F_t$, $H_{F_t}$ is the mean curvature of the immersion $F_t$, $\nu_{F_t}$ is the unit normal vector field to $F_t$, and $\cH^n_h$ is the $n$-dimensional Hausdorff measure with respect to the metric $h$. 

\section{Conformal Change for Stable Minimal Hypersurfaces}\label{sec:GL}
We study the geometry of stable minimal hypersurfaces in Euclidean space under a conformal change introduced by Gulliver--Lawson \cite{GL:conf} to study singularities of stable minimal hypersurfaces.

Let $F: M^n \to \R^{n+1}$ be a complete, two-sided stable minimal immersion. Let $g$ be the pullback metric on $M$. Let $r$ denote the Euclidean distance from $0$ in $\R^{n+1}$. Consider the conformal metric $\tilde{g} = r^{-2}g$ on $N = M \setminus F^{-1}(\{0\})$, as in \cite{GL:conf}. Note that $(N, \tilde{g})$ is complete.
Henceforth, we use tildes to denote quantities with respect to $\tilde{g}$; and otherwise we use the metric $g$. For instance, we let $d\mu$ and $d\tilde{\mu}$ denote the volume measures on $N$ with respect to $g$ and $\tilde{g}$ respectively.

We prove here the following result.
\begin{theorem}\label{thm:spectralbiRic}
    Suppose $n = 4$. Then there is a smooth function $V$ on $(N, \tilde{g})$ so that
    \[ V \geq 1 - \tilde{\lambda}_{\BiRic} \]
    and
    \[ \int_N |\tilde{\nabla} \psi|_{\tilde{g}}^2 d\tilde{\mu} \geq \int_N V\psi^2 d\tilde{\mu} \]
    for all $\psi \in C^{\infty}_{c}(N)$.
\end{theorem}

We note that Theorem \ref{thm:spectralbiRic} shows that $(N, \tilde{g})$ satisfies a weak form of the condition $\widetilde{\BiRic} \geq 1$.

\subsection{Standard calculations for Euclidean hypersurfaces}
We begin with some standard calculations about the Euclidean distance function $r$ on immersed hypersurfaces.

\begin{proposition}\label{prop:hessr}
    Let $M^n \to \R^{n+1}$ be an immersion with unit normal vector field $\nu$. Then
    \[ \Hess^M r = r^{-1}g - r^{-1}dr\otimes dr - r^{-1}\langle \vec{x}, \nu\rangle A. \]
\end{proposition}
\begin{proof}
    We compute
    \[ \partial_i r = \frac{x_i}{r},\ \ \partial_i\partial_jr = \frac{\delta_{ij}}{r} - \frac{x_ix_j}{r^3}
    \]
    Thus
    \[
    \Hess^{\R^{n+1}} r = r^{-1}g_{\text{Euc}} - r^{-1}dr\otimes dr. \]
    Hence, for $X$ and $Y$ tangent to $M$, we have
    \begin{align*}
        \Hess^M r(X, Y) & = \langle \nabla^M_X\nabla^M r, Y \rangle\\
        & = \langle \nabla^{\R^{n+1}}_X\nabla^{\R^{n+1}} r, Y \rangle - \langle \nabla^{\R^{n+1}}_X(\nabla^{\R^{n+1}} r)^{\perp}, Y \rangle\\
        & = \Hess^{\R^{n+1}} r(X, Y) + \left\langle \frac{\vec{x}^{\perp}}{r}, (\nabla_X^{\R^{n+1}}Y)^{\perp}\right\rangle\\
        & = r^{-1}g(X,Y) - r^{-1}(dr \otimes dr)(X,Y) - r^{-1}\langle \vec{x}, \nu\rangle A(X, Y),
    \end{align*}
    which concludes the proposition.
\end{proof}

We write $\phi = -\log r$ so that $\tilde g = e^{2\phi}g$.

\begin{proposition}\label{prop:hesslogr}
    Let $M^n \to \R^{n+1}$ be an immersion with unit normal vector field $\nu$. Then
    \[ \Hess^M (\log r) = r^{-2}g - 2r^{-2}dr \otimes dr - r^{-2}\langle \vec{x}, \nu\rangle A. \]
\end{proposition}
\begin{proof}
    For any $f > 0$, we compute
    \[ \Hess^M (\log f)(X,Y) = \left\langle \nabla^M_X\left(\frac{\nabla^M f}{f}\right), Y\right\rangle = f^{-1}\Hess^M f(X, Y) - f^{-2}(df \otimes df)(X, Y). \]
    We conclude by Proposition \ref{prop:hessr}.
\end{proof}

\subsection{Curvature in the conformal metric}
We now compute relevant curvature quantities related to the conformal metric $\tilde{g}$.

First, we set up a convenient orthonormal basis for $T_pM$ with respect to the metrics $g$ and $\tilde{g}$. Let $\{e_i\}_{i=1}^n$ be an orthonormal basis for $T_pM$ with respect to $g$. Then $\{\tilde{e}_i = re_i\}_{i=1}^n$ is an orthonormal basis for $T_pM$ with respect to $\tilde{g}$.

We are now equipped to compute the sectional curvatures of $(N, \tilde{g})$. We will write $R_{ijji} = \Rs(e_i,e_j,e_j,e_i)$, $A_{ij} = A(e_i,e_j)$, and $\tilde{R}_{ijji} = \tilde{\Rs}(\tilde e_i,\tilde e_j,\tilde e_j,\tilde e_i)$.

\begin{proposition}\label{prop:confcurv}
    In the above frame, we have
    \[ r^2R_{ijji} = \tilde{R}_{ijji} - 2 + |dr|^2 + (dr(e_i))^2 + (dr(e_j))^2 + \langle \vec{x}, \nu\rangle(A_{ii}+A_{jj}). \]
\end{proposition}
\begin{proof}
    Following \cite[Theorem 7.30]{Lee}, we define the tensor
    \[ T = \Hess^M \phi - d\phi \otimes d\phi + \frac{1}{2}|d\phi|^2g. \]
    As above, we write $T_{ij} = T(e_i,e_j)$. By Proposition \ref{prop:hesslogr}, we have
    \begin{equation}\label{eqn:T}
        T = -r^{-2}g + r^{-2}dr \otimes dr + \frac{1}{2}r^{-2}|dr|^2g + r^{-2}\langle \vec{x}, \nu \rangle A.
    \end{equation}
    Using the formula for the Riemann curvature tensor under a conformal change \cite[Theorem 7.30]{Lee} and \eqref{eqn:T}, we compute
    \begin{align*}
        \tilde{R}_{ijji} & = \tilde{\Rs}(\tilde{e}_i, \tilde{e}_j, \tilde{e}_j, \tilde{e}_i) = r^4\tilde{\Rs}(e_i, e_j, e_j, e_i)\\
        & = r^2(R_{ijji} - T_{ii} - T_{jj})\\
        & = r^2R_{ijji} + 2 - |dr|^2 - (dr(e_i))^2 - (dr(e_j))^2 - \langle \vec{x}, \nu\rangle (A_{ii} + A_{jj}),
    \end{align*}
    which concludes the proposition.
\end{proof}

Taking the appropriate combinations of sectional curvatures, we use Proposition \ref{prop:confcurv} to compute the bi-Ricci curvatures of $(N, \tilde{g})$ in the case where the immersion is minimal.

\begin{proposition}\label{prop:confbiRic}
    Let $M^n \to \R^{n+1}$ be a minimal immersion with unit normal vector field $\nu$. Then we have
    \begin{align*}
        r^2\BiRic(e_1, e_2) & = \widetilde{\BiRic}(\tilde{e}_1, \tilde{e}_2) - (4n-6) + (2n-1)|dr|^2\\
        & \hspace{0.5cm} + (n-3)(dr(e_1)^2 + dr(e_2)^2)\\
        & \hspace{0.5cm} + (n-3)\langle\vec{x}, \nu\rangle(A_{11} + A_{22}).
    \end{align*}
\end{proposition}
\begin{proof}
    Using Proposition \ref{prop:confcurv}, we compute
    \begin{align*}
        r^2\BiRic(e_1, e_2) & = \sum_{i=2}^n r^2R_{1ii1} + \sum_{j=3}^n r^2R_{2jj2}\\
        & = \widetilde{\BiRic}(\tilde{e}_1, \tilde{e}_2) - (4n-6) + (2n-3)|dr|^2\\
        & \hspace{0.5cm} + 2|dr|^2 + (n-3)(dr(e_1)^2 + dr(e_2)^2)\\
        & \hspace{0.5cm} + 2\langle\vec{x}, \nu\rangle \Tr(A) + (n-3)\langle\vec{x}, \nu\rangle(A_{11} + A_{22})\\
        & = \widetilde{\BiRic}(\tilde{e}_1, \tilde{e}_2) - (4n-6) + (2n-1)|dr|^2\\
        & \hspace{0.5cm} + (n-3)(dr(e_1)^2 + dr(e_2)^2)\\
        & \hspace{0.5cm} + (n-3)\langle\vec{x}, \nu\rangle(A_{11} + A_{22}).
    \end{align*}
    In the last equality we used that the trace of $A$ is zero for a minimal hypersurface.
\end{proof}

To exploit the stability inequality, we use the Gauss equation to express $\BiRic$ in terms of the second fundamental form of $M$.

\begin{proposition}\label{prop:gausseqn}
    Let $M^n \to \R^{n+1}$ be a minimal immersion. Then
    \[ \BiRic(e_1, e_2) = -\sum_{i=1}^n A_{1i}^2 - \sum_{j=2}^n A_{2j}^2 - A_{11}A_{22}. \]
\end{proposition}
\begin{proof}
    Using the Gauss equation and $\Tr A = 0$, we compute
    \begin{align*}
        \BiRic(e_1, e_2)
        & = \sum_{i=2}^n R_{1ii1} + \sum_{j=3}^n R_{2jj2}\\
        & = \sum_{i=2}^n (A_{11}A_{ii} - A_{1i}^2) + \sum_{j=3}^n (A_{22}A_{jj} - A_{2j}^2)\\
        & = -\sum_{i=1}^n A_{1i}^2 - \sum_{j=2}^n A_{2j}^2 - A_{11}A_{22}.
    \end{align*}
This completes the proof. \end{proof}

Choose the basis vectors $e_1$ and $e_2$ so that $\widetilde{\BiRic}(\tilde{e}_1, \tilde{e}_2) = \tilde{\lambda}_{\BiRic}$. We can now bound $|A|^2$ in terms of $\tilde{\lambda}_{\BiRic}$.

\begin{proposition}\label{prop:r2a2}
    Let $M^n \to \R^{n+1}$ be a minimal immersion with unit normal vector field $\nu$. For $n \geq 3$, we have
    \[ r^2|A|^2 \geq \frac{2}{n-2}\left((3n-3) - (2n-1)|dr|^2 - \tilde{\lambda}_{\BiRic}\right). \]
\end{proposition}
\begin{proof}
    Combining Propositions \ref{prop:confbiRic} and \ref{prop:gausseqn}, we compute
    \begin{align}\label{eqn:combo1}
        r^2\Big(\sum_{i=1}^n A_{1i}^2 & + \sum_{j=2}^n A_{2j}^2 + A_{11}A_{22}\Big) + (n-3)\langle \vec{x}, \nu\rangle(A_{11} + A_{22})\\
        & = (4n-6) - (2n-1)|dr|^2 - (n-3)(dr(e_1)^2 + dr(e_2)^2) - \tilde{\lambda}_{\BiRic}.\notag
    \end{align}
    Using $\langle \vec{x}, \nu\rangle = rdr(\nu)$ and Young's inequality, we have
    \[ |(n-3)\langle \vec{x}, \nu\rangle(A_{11} + A_{22})| \leq (n-3)dr(\nu)^2 + \frac{n-3}{4}r^2(A_{11} + A_{22})^2. \]
    Combined with \eqref{eqn:combo1} and the fact that $dr(e_1)^2 + dr(e_2)^2 + dr(\nu)^2 \leq 1$, we find
    \begin{align}\label{eqn:newineq1}
        r^2\Big(\sum_{i=1}^n A_{1i}^2 + \sum_{j=2}^n A_{2j}^2 + A_{11}A_{22} + & \frac{n-3}{4}(A_{11}+A_{22})^2\Big)\\
        & \geq (3n-3) - (2n-1)|dr|^2 - \tilde{\lambda}_{\BiRic}.\notag
    \end{align}
    Now we compute (using the fact that $\Tr A = 0$), with $\sigma \in (0, 1)$ arbitrary,
    \begin{align*}
        A_{11}^2 + A_{22}^2 & + A_{11}A_{22} + \frac{n-3}{4}(A_{11}+A_{22})^2
        = \frac{1}{2}(A_{11}^2 + A_{22}^2) + \frac{n-1}{4}(A_{11}+A_{22})^2\\
        & = \frac{1}{2}(A_{11}^2+A_{22}^2) + \frac{n-1}{4}\sigma(A_{11}+A_{22})^2 + \frac{n-1}{4}(1-\sigma)(A_{33} + \hdots + A_{nn})^2\\
        & \leq \left(\frac{1}{2} + \frac{n-1}{2}\sigma\right)(A_{11}^2 + A_{22}^2) + \frac{(n-1)(n-2)}{4}(1-\sigma)(A_{33}^2 + \hdots + A_{nn}^2).
    \end{align*}
    Taking $\sigma = \frac{n-3}{n-1}$, we have
    \begin{equation}\label{eqn:newineq2}
        A_{11}^2 + A_{22}^2 + A_{11}A_{22} + \frac{n-3}{4}(A_{11}+A_{22})^2
        \leq \frac{n-2}{2}(A_{11}^2 + \hdots + A_{nn}^2).
    \end{equation}
    Hence, for $n \geq 3$, we have
    \begin{align*}
        \frac{n-2}{2}r^2|A|^2
        & \geq r^2\Big(\frac{n-2}{2}\sum_{i=1}^n A_{ii}^2 + \sum_{i=2}^n A_{1i}^2 + \sum_{j=3}^n A_{2j}^2\Big)\\
        & \geq r^2\Big(\sum_{i=1}^n A_{1i}^2 + \sum_{j=2}^n A_{2j}^2 + A_{11}A_{22} + \frac{n-3}{4}(A_{11}+A_{22})^2\Big)\\
        & \geq (3n-3) - (2n-1)|dr|^2 - \tilde{\lambda}_{\BiRic},
    \end{align*}
    where we used \eqref{eqn:newineq2} in the second inequality and \eqref{eqn:newineq1} in the third inequality.
\end{proof}

\subsection{Stability inequality in the conformal metric}

\begin{proposition}\label{prop:conflaplogr} Let $M^n \to \R^{n+1}$ be a minimal immersion. We have
\[ \tilde{\Delta}^M (\log r) = n - n|dr|^2. \]
\end{proposition}
\begin{proof}
    By the formula for the Laplace--Beltrami operator under a conformal transformation (see for instance \cite[Lemma 2.1]{GL:conf}) and Proposition \ref{prop:hesslogr} (using $\mathrm{Tr}A = 0$), we have
    \begin{align*}
        \tilde{\Delta}^M(\log r)
        & = r^2(\Delta^M (\log r) - (n-2)r^{-2}|dr|^2)\\
        & = n - 2|dr|^2 - (n-2)|dr|^2 = n - n|dr|^2,
    \end{align*}
    as desired.
\end{proof}

We can now rewrite the stability inequality in the metric $\tilde{g}$.

\begin{proposition}\label{prop:confstability}
    Let $M^n \to \R^{n+1}$ be a two-sided stable minimal immersion. Then for all $\psi \in C^{\infty}_c(N)$ we have
    \[ \int_N |\tilde{\nabla}\psi|_{\tilde{g}}^2d\tilde{\mu} \geq \int_N\Big(r^2|A|^2 -\frac{n(n-2)}{2} + \Big(\frac{n(n-2)}{2} - \frac{(n-2)^2}{4}\Big)|dr|^2\Big)\psi^2d\tilde{\mu}. \]
\end{proposition}
\begin{proof}
    Using
    \[ d\tilde{\mu} = r^{-n}d\mu \ \ \text{and}\ \ |\tilde{\nabla}f|^2_{\tilde{g}} = r^2|\nabla f|^2, \]
    the stability inequality for $M$ can be written as
    \[ \int_N r^{n-2}|\tilde{\nabla}f|_{\tilde{g}}^2d\tilde{\mu} \geq \int_Nr^{n-2}(r^2|A|^2)f^2d\tilde{\mu} \]
    for any $f \in C^{\infty}_c(N)$. We take $f = r^{\frac{2-n}{2}}\psi$ for $\psi \in C^{\infty}_c(N)$. Then
    \[ \tilde{\nabla} f = r^{\frac{2-n}{2}}\tilde{\nabla}\psi - \frac{n-2}{2}r^{-\frac{n}{2}}\psi\tilde{\nabla}r. \]
    In particular,
    \[
    |\tilde{\nabla} f|_{\tilde{g}}^2 = a+b+c
    \]
    where 
    \[
    a:=r^{2-n}|\tilde{\nabla}\psi|_{\tilde{g}}^2, \quad b := \frac{(n-2)^2}{4}r^{-n}\psi^2|\tilde{\nabla}r|_{\tilde{g}}^2, \quad c := - (n-2)r^{1-n}\psi\langle\tilde{\nabla}\psi, \tilde{\nabla}r\rangle_{\tilde{g}}.
    \]
    We have
    \[ \int_N r^{n-2}a\, d\tilde{\mu} = \int_N |\tilde{\nabla}\psi|^2_{\tilde{g}}d\tilde{\mu}. \]
    Since $r^{-2}|\tilde{\nabla}r|^2_{\tilde{g}} = |dr|^2$, we have
    \[ \int_N r^{n-2}b\, d\tilde{\mu} = \int_N \frac{(n-2)^2}{4}|dr|^2\psi^2d\tilde{\mu}. \]
    Finally, we use integration by parts and Proposition \ref{prop:conflaplogr} to compute
    \begin{align*}
        \int_N r^{n-2}c\, d\tilde{\mu} & = -\int_N \frac{n-2}{2}\langle \tilde{\nabla}(\psi^2), \tilde{\nabla}(\log r)\rangle_{\tilde{g}}d\tilde{\mu}\\
        & = \int_N \frac{n-2}{2}\tilde{\Delta}(\log r) \psi^2d\tilde{\mu}\\
        & = \int_N \Big(\frac{n(n-2)}{2} - \frac{n(n-2)}{2}|dr|^2\Big)\psi^2d\tilde{\mu}.
    \end{align*}
    The assertion follows from the above expressions.
\end{proof}

We now rephrase our estimate for $r^2|A|^2$ from Proposition \ref{prop:r2a2} to suit this form of the stability inequality.

\begin{proposition}\label{prop:confconclusion}
    Let $M^n \to \R^{n+1}$ be a minimal immersion. For $3 \leq n \leq 5$, we have
    \[ r^2|A|^2 -\frac{n(n-2)}{2} + \Big(\frac{n(n-2)}{2} - \frac{(n-2)^2}{4}\Big)|dr|^2 \geq \frac{2}{n-2}\Big(\frac{(2-n)(n^2-4n-4)}{8} - \tilde{\lambda}_{\BiRic}\Big). \]
\end{proposition}
\begin{proof}
    By Proposition \ref{prop:r2a2}, we have
    \begin{align*}
        r^2|A|^2 & -\frac{n(n-2)}{2} + \Big(\frac{n(n-2)}{2} - \frac{(n-2)^2}{4}\Big)|dr|^2\\
        & \geq \frac{6(n-1)}{n-2} - \frac{n(n-2)}{2} + \Big(\frac{n(n-2)}{2} - \frac{(n-2)^2}{4} - \frac{2(2n-1)}{n-2}\Big)|dr|^2 - \frac{2}{n-2}\tilde{\lambda}_{\BiRic}.
    \end{align*}
    Note that the coefficient of $|dr|^2$ on the right-hand side is negative for $3 \leq n \leq 5$, so we can use $|dr|^2 \leq 1$ to conclude that the left-hand side is greater than or equal to
    \[ \frac{2}{n-2}\Big(\frac{(2-n)(n^2-4n-4)}{8} - \tilde{\lambda}_{\BiRic}\Big), \]
    as desired.
\end{proof}

Theorem \ref{thm:spectralbiRic} follows by plugging $n=4$ into Proposition \ref{prop:confconclusion} and applying the stability inequality as formulated in Proposition \ref{prop:confstability}.

\section{$\mu$-Bubbles in Spectral Positive Bi-Ricci Curvature}\label{sec:mu}
We generalize the $\mu$-bubble construction in uniformly positive bi-Ricci curvature of \cite{xu:biRic} to manifolds with spectral uniformly positive bi-Ricci curvature. Our main tool is the notion of \emph{warped} $\mu$-bubbles (see \cite{CL:soapbubbles}), extending the standard $\mu$-bubbles used in \cite{xu:biRic}. Roughly speaking, a $\mu$-bubble is the boundary of a finite perimeter set minimizing a prescribed mean curvature-type functional (meaning a functional of the form ``area of the boundary'' plus ``interior integral of a potential function''). These minimizing hypersurfaces behave similarly to stable minimal hypersurfaces, as long as the potential function is judicially chosen.

Suppose that $(N^n, \g)$ is a smooth complete noncompact Riemannian manifold that admits a smooth function $V$ so that
\[ V \geq 1 - \lambda_{\BiRic}(\g) \]
and
\[ \int_N |\nabla \psi|^2 \geq \int_N V\psi^2 \]
for all $\psi \in C^{\infty}_c(N)$. The reader should note that this condition holds for the conformal metric on the stable minimal hypersurface by Theorem \ref{thm:spectralbiRic}. However, the subsequent estimates hold for any such manifold. Recall from \cite[Theorem 1]{FCS:stable} that there is a positive function $u$ on $N$ satisfying
\begin{equation}\label{eqn:u}
    -\Delta^Nu = Vu \geq (1-\lambda_{\BiRic}(\g))u.
\end{equation}

We prove the following theorem about such $N$ in dimension $n=4$.

\begin{theorem}\label{thm:mububbles}
    Let $X \subset N^4$ be a closed subset with smooth boundary $\partial X = \partial_+X \sqcup \partial_-X$ for some nonempty smooth hypersurfaces $\partial_{\pm}X$. Suppose $d_N(\partial_+X, \partial_-X) \geq 10\pi$. Then there is a connected, relatively open subset $\Omega \subset X$ with smooth boundary $\partial \Omega = \partial_-X \sqcup \Sigma$ so that
    \begin{itemize}
        \item $\partial_-X \subset \Omega$,
        \item $\Sigma \subset X\setminus \partial X$ is a closed submanifold
        \item $\Omega \subset B_{10\pi}(\partial_-X)$, and
        \item there is a smooth function $W \in C^{\infty}(\Sigma)$ so that
        \[ W \geq \frac{3}{4}\left(\frac{1}{2} - \lambda_{\Ric}(\Sigma)\right) \]
        and
        \[ \int_{\Sigma} |\nabla \psi|^2 \geq \int_{\Sigma} W\psi^2 \]
        for all $\psi \in C^{\infty}(\Sigma)$.
    \end{itemize}
\end{theorem}

\begin{remark}
    The hypersurface $\Sigma$ in Theorem \ref{thm:mububbles} is referred to as a (\emph{warped}) \emph{$\mu$-bubble}.
\end{remark}

The strategy to prove Theorem \ref{thm:mububbles} is to minimize a certain warped prescribed mean curvature functional.  

Let $w$ be a smooth positive function on $N$. Let $\Omega \subset N$ be an open set with smooth boundary (or more generally a set of finite perimeter). Let $\nu$ denote the outward unit normal vector field along $\partial \Omega$. Let $h$ be a smooth function defined in a neighborhood of $\partial \Omega$. We study minimizers of the warped prescribed mean curvature functional
\[ \cA(\Omega) = \int_{\partial \Omega} w d\cH^{n-1} - \int_{\Omega} hw d\cH^n. \]
Ultimately, we will take $w = u$ to be our warping function, but we leave $w$ general for most of the calculations. Note that this functional may be viewed as a $\mu$-bubble functional on a warped manifold, see \cite[Remark 11]{CL:soapbubbles} and \cite{xu:biRic} for similar computations.

\subsection{First variation formula}
The first variation formula for $\cA$ can be computed as follows (see for instance \cite[Lemma 13]{CL:soapbubbles}).
\begin{proposition}\label{prop:firstvar}
    Let $\{\Omega_t\}_{|t|<\eps}$ be a smooth family of open sets with smooth boundary, where $\Omega_0 = \Omega$ and the variation vector field is $V_t$. Then
    \[ \frac{d}{dt} \cA(\Omega_t) = \int_{\partial \Omega_t} \langle \nabla^Nw,V_t^\perp\rangle + wH_t\langle \nu_t, V_t\rangle - wh\langle \nu_t, V_t\rangle d\cH^{n-1}, \]
    where $\nu_t$ denotes the outward unit normal vector field along $\partial \Omega_t$ and $H_t$ denotes the scalar mean curvature of $\partial \Omega_t$ with respect to $\nu_t$. Hence, critical points for $\cA$ satisfy
    \[ H = h - w^{-1}\langle \nabla^Nw, \nu\rangle. \]
\end{proposition}

\subsection{Second variation formula}
We now prove the following second variation formula (see for instance \cite[Lemma 14]{CL:soapbubbles}). Although the theorem statement below only holds in dimension $n = 4$, we carry out the computations for general $n$ until plugging in $n=4$ at the very end of the proof.

\begin{theorem}\label{thm:secondvar}
    Let $\Omega \subset N^4$ be an open set with smooth boundary that is a stable critical point for $\cA$ with weight function $w=u$. Let $\Sigma = \partial \Omega$. Let $\gamma$ denote the pullback metric on $\Sigma$. Then there is a smooth function $W \in C^{\infty}(\Sigma)$ so that
    \[ W \geq \frac{3}{4}\left(\frac{1}{2}-\lambda_{\Ric}(\gamma)\right) \]
    and
    \[ \int_{\Sigma} |\nabla^{\Sigma}\psi|^2 \geq \int_{\Sigma} W\psi^2 + \frac{3}{8} \int_{\Sigma} \left(1 + h^2 - 2|\nabla^Nh|\right)\psi^2 \]
    for all $\psi \in C^{\infty}_c(\Sigma)$.
\end{theorem}
\begin{proof}
    We take a variation $\{\Omega_t\}$ where $\Omega_0 = \Omega$ is a critical point for $\cA$. We can choose our variation so that $D_t V_t = 0$ by taking the normal exponential flow of $V_0 = \phi\nu$.

    We compute using Proposition \ref{prop:firstvar}
    \begin{align*}
        \frac{d^2}{dt^2}\Big|_{t=0}\cA(\Omega_t)
        & = \int_{\Sigma} \phi^2\Hess^Nw(\nu,\nu) - \phi\langle \nabla^{\Sigma}w, \nabla^{\Sigma}\phi\rangle + \phi^2\langle \nabla^N w, \nu\rangle H\\
        & \hspace{0.5cm} -\int_{\Sigma} w(\phi \Delta^{\Sigma}\phi + (|A_{\Sigma}|^2 + \Ric_{\g}(\nu,\nu))\phi^2)\\
        & \hspace{0.5cm} - \int_{\Sigma} \phi^2\langle \nabla^Nw, \nu\rangle h + \phi^2w\langle \nabla^Nh, \nu\rangle.
    \end{align*}
    We use integration by parts on the $-w\phi\Delta^{\Sigma}\phi$ term and the formula
    \[ \Hess^N f(\nu, \nu) = \Delta^Nf - \Delta^{\Sigma}f - \langle \nabla^Nf, \nu\rangle H. \]
    This gives
    \begin{align*}
        \frac{d^2}{dt^2}\Big|_{t=0}\cA(\Omega_t)
        & = \int_{\Sigma} \phi^2(\Delta^Nw - \Delta^{\Sigma}w)\\
        & \hspace{0.5cm} +\int_{\Sigma} w(|\nabla^{\Sigma} \phi|^2 - (|A_{\Sigma}|^2 + \Ric_{\g}(\nu,\nu))\phi^2)\\
        & \hspace{0.5cm} - \int_{\Sigma} \phi^2\langle \nabla^Nw, \nu\rangle h -  \int_{\Sigma} \phi^2w\langle \nabla^Nh, \nu\rangle.
    \end{align*}
    Since $\Omega$ is a stable critical point of $\cA$, we thus have
    \begin{align}\label{eqn:middle}
        \int_{\Sigma} w|\nabla^{\Sigma}\phi|^2 - \phi^2\Delta^{\Sigma}w & \geq \int_{\Sigma} (-\Delta^Nw + (|A_{\Sigma}|^2 + \Ric_{\g}(\nu, \nu))w)\phi^2\\
        & \hspace{0.5cm} + \int_{\Sigma} ( h\langle \nabla^Nw,\nu\rangle+ w\langle \nabla^Nh,\nu\rangle)\phi^2.\notag
    \end{align}
    Take $\phi = w^{-1/2}\psi$. We compute
    \[ \nabla^{\Sigma} \phi = w^{-1/2}\nabla^{\Sigma} \psi - \frac{1}{2}w^{-3/2}\psi \nabla^{\Sigma} w. \]
    Write
    \[
    w|\nabla^{\Sigma} \phi|^2 = a+ b + c 
    \]
    where
    \[ a:= |\nabla^{\Sigma} \psi|^2, \quad b : = - w^{-1}\psi\langle \nabla^{\Sigma} w, \nabla^{\Sigma} \psi\rangle, \quad \textrm{and} \quad c:= \frac{1}{4}w^{-2}\psi^2|\nabla^{\Sigma} w|^2. \]
    We have
    \begin{align*}
        \int_{\Sigma} a & = \int_{\Sigma} |\nabla^{\Sigma} \psi|^2,\\
        \int_{\Sigma} c & = \frac{1}{4} \int_{\Sigma}|\nabla^{\Sigma}\log w|^2\psi^2, 
    \end{align*}  
    and
    \begin{align*}
        \int_{\Sigma} b - \phi^2\Delta^{\Sigma}w
        & = \int_{\Sigma} w^{-1}\psi\langle \nabla^{\Sigma} w, \nabla^{\Sigma} \psi\rangle - \int_{\Sigma} \psi^2|\nabla^{\Sigma}\log w|^2\\
        & \leq \left(\frac{\eps}{2}-1\right)\int_{\Sigma} \psi^2|\nabla^{\Sigma}\log w|^2 + \frac{1}{2\eps}\int_{\Sigma} |\nabla^{\Sigma}\psi|^2
    \end{align*}
    for all $\eps>0$. 
    Taking $\eps = 3/2$, we have
    \[ \int_{\Sigma} w|\nabla^{\Sigma}\phi|^2 - \phi^2\Delta^{\Sigma}w \leq \frac{4}{3}\int_{\Sigma} |\nabla^{\Sigma}\psi|^2. \]
    Combined with \eqref{eqn:middle} we have
    \begin{align}\label{eqn:middle2}
        \frac 4 3 \int_{\Sigma} |\nabla^{\Sigma}\psi|^2
        & \geq \int_{\Sigma}\left(-\frac{\Delta^Nw}{w} + |A_{\Sigma}|^2 + \Ric_{\g}(\nu, \nu)-\frac{1}{2}H^2 - \frac{1}{2}\right)\psi^2\\
        & \hspace{0.5cm} + \int_{\Sigma} \left(\frac{1}{2} + \frac{1}{2}H^2 + h\langle \nabla^N \log w, \nu\rangle + \langle \nabla^Nh, \nu\rangle\right)\psi^2.\notag
    \end{align}
    Since $\Sigma$ is a critical point for $\cA$, we have
    \[ H^2 = h^2 + \langle \nabla^N \log w, \nu\rangle^2 - 2h\langle \nabla^N\log w, \nu\rangle \geq h^2 - 2h\langle \nabla^N\log w, \nu\rangle. \]
    Hence, we have 
    \begin{align}\label{eqn:middle3}
        \frac 4 3\int_{\Sigma} |\nabla^{\Sigma}\psi|^2
        & \geq \int_{\Sigma} \left(-\frac{\Delta^Nw}{w} + |A_{\Sigma}|^2 + \Ric_{\g}(\nu, \nu)-\frac{1}{2}H^2 - \frac{1}{2}\right)\psi^2\\
        & \hspace{0.5cm} + \frac 12 \int_{\Sigma} \left(1 + h^2 - 2|\nabla^Nh|\right)\psi^2.\notag
    \end{align}

    It remains to find a good lower bound for the first integrand on the right-hand side of \eqref{eqn:middle3}. Using the Gauss equation, we compute
    \begin{align*}
        \Ric_\gamma(e_1, e_1)
        & = \sum_{i=1}^{n-1} \Rs_\gamma(e_1, e_i, e_i, e_1)\\
        & = \sum_{i=1}^{n-1} (\Rs_{\g}(e_1, e_i, e_i, e_1) + A_{11}A_{ii} - A_{1i}^2)\\
        & = \BiRic_{\g}(e_1, \nu) - \Ric_{\g}(\nu, \nu) + A_{11}\sum_{i=2}^{n-1} A_{ii} - \sum_{i=2}^{n-1}A_{1i}^2.
    \end{align*}
    Moreover, using $\Tr(A) = H$, we have (for all $\eta>0$)
    \begin{align*}
        A_{11}\sum_{i=2}^{n-1} A_{ii}
        = -A_{11}^2 + A_{11}H
        & = -A_{11}^2 - H\sum_{i=2}^{n-1}A_{ii} + H^2\\
        & \geq -A_{11}^2 - \frac{1}{2\eta}\left(\sum_{i=2}^{n-1}A_{ii}\right)^2 + \left(1 - \frac{\eta}{2}\right)H^2\\
        & \geq -A_{11}^2 - \frac{n-2}{2\eta}\sum_{i=2}^{n-1}A_{ii}^2 + \left(1 - \frac{\eta}{2}\right)H^2\\
        & = -\sum_{i=1}^{n-1}A_{ii}^2 + \frac{6-n}{4}H^2,
    \end{align*}
    where we took $\eta = \frac{n-2}{2}$ in the last line.
    Choosing $e_1$ so that $\lambda_{\Ric}(\gamma) = \Ric_{\gamma}(e_1, e_1)$, we have
    \[ |A_\Sigma|^2 + \Ric_{\g}(\nu,\nu) \geq \lambda_{\BiRic}(\g) - \lambda_{\Ric}(\gamma) + \frac{6-n}{4}H^2. \]
    Taking $n = 4$ and $w = u$, we have
    \[ -\frac{\Delta^Nw}{w} + |A_{\Sigma}|^2 + \Ric_{\g}(\nu, \nu)-\frac{1}{2}H^2 - \frac{1}{2} \geq \frac{1}{2} - \lambda_{\Ric}(\gamma). \]
    Taking $W$ to be $\frac{3}{4}$ times the left-hand side above (which is smooth) completes the proof.
\end{proof}

\begin{proof}[Proof of Theorem \ref{thm:mububbles}]
Equipped with the second variation formula from Theorem \ref{thm:secondvar}, the proof now follows by taking $h$ to be the standard $\mu$-bubble prescribing function, chosen precisely so that $1 + h^2 - 2|\nabla^Nh| \geq 0$,  and then minimizing the functional $\cA$ (see \cite[Lemma 24]{CL:aniso} for the choice of $h$ and see \cite[Proposition 12]{CL:soapbubbles} or \cite[Proposition 2.1]{zhu:mububbles} for the existence theory).

For completeness, we construct the function $h$. Let $\varphi_0$ be a smoothing of the function $d(\partial_-X,\,\cdot\,)$ so that $|\nabla^N\varphi_0|\leq 2$ and $\left.\varphi_0\right|_{\partial_-X} \equiv 0$. Let $\eps \in (0, 1/2)$ so that $\eps$ and $4\pi + 2\eps$ are regular values of $\varphi_0$. Define
\[ \varphi = \frac{\varphi_0 - \eps}{4 + \frac{\eps}{\pi}} - \frac{\pi}{2}. \]
Then $|\nabla^N \varphi| \leq \frac{1}{2}$, and the set $\Omega_1 = \{-\pi/2 < \varphi < \pi/2\}$ has smooth boundary and satisfies $\Omega_1 \subset B_{10\pi}(\partial_-X)$. On $\Omega_1$, we define
\[ h = -\tan(\varphi). \]
Since
\[ \nabla^N h = -(1+\tan^2(\varphi))\nabla^N \varphi = -(1 + h^2)\nabla^N \varphi, \]
it holds that
\[ 2|\nabla^N h| \leq 1 + h^2. \]
Hence, $h$ has the desired property; the rest of the proof follows as in \cite[Lemma 24]{CL:aniso}.
\end{proof}

\section{Geometric Estimates for $\mu$-Bubbles}\label{sec:mu-geo}
We prove that the $\mu$-bubbles constructed in the previous section have uniformly bounded diameter and volume.

\begin{theorem}\label{thm:diameter_and_vol}
    Suppose that $(\Sigma^3, \gamma)$ is a connected, closed Riemannian 3-manifold that admits a smooth function $W$ and a constant $\alpha \in (0, 2]$ so that
    \[ W \geq \alpha^{-1}(2 - \lambda_{\Ric}(\gamma)) \]
    and
    \[ \int_{\Sigma} |\nabla \psi|^2 \geq \int_{\Sigma} W\psi^2 \]
    for all $\psi \in C^{\infty}(\Sigma)$. Then
    \[ \mathrm{diam}(\Sigma, \gamma) \leq \pi  \]
    and
    \[ \mathrm{Vol}(\Sigma, \gamma) \leq 2\pi^2, \]
    where $\mathrm{diam}$ and $\mathrm{Vol}$ denote the diameter and volume, respectively, of $(\Sigma, \gamma)$.
\end{theorem}
Note that both inequalities are sharp for the unit $3$-sphere in $\R^4$ (see also Remark \ref{rema:rigid-vol} explaining the rigidity statement for the volume estimate). 
\begin{proof}[Proof of the diameter bound]
    A smooth positive first eigenfunction $\theta$ of the operator $-\Delta - W$ satisfies
    \begin{equation}\label{eqn:eigenfunction}
        -\Delta \theta \geq \alpha^{-1}(2 - \lambda_{\Ric}(\gamma))\theta.
    \end{equation}
    Hence,
    \[ \Ric^{(\theta, \alpha)}_\gamma \equiv \Ric_{\gamma} - \alpha(\theta^{-1}\Delta \theta)\gamma \geq \Ric_\gamma - \lambda_{\Ric}(\gamma)\gamma + 2\gamma \geq 2\gamma, \]
    where we adapt the notation of \cite{SY:general}. Since $\alpha \leq \frac{4}{3-1} = 2$, we have 
    \[ \mathrm{diam}(\Sigma, \gamma) \leq \pi, \]
    by \cite[Corollary 1]{SY:general}. This completes the proof of the diameter bound.
\end{proof}

It remains to prove the volume bound for $\Sigma$. The strategy is to exploit the concavity properties of a weighted isoperimetric profile. These arguments extend the strategy of Bray's proof of the Bishop volume comparison from \cite{Bray:thesis} (we follow the exposition in \cite{Brendle:scal}).

\subsection{Weighted isoperimetric profile}
For an open set $\Omega \subset \Sigma$ with smooth boundary, we define a weighted area and volume functional by
\[ a(\Omega) = \int_{\partial\Omega} \theta^{\alpha}\ d\cH^2_\gamma\ \text{and}\ \ v(\Omega) = \int_{\Omega} \theta^{\alpha}\ d\cH^3_\gamma, \]
where $\theta$ is the unique positive first eigenfunction of $-\Delta - W$ from \eqref{eqn:eigenfunction} with $\min \theta = 1$.

The \emph{weighted isoperimetric profile} is the function $\cI : (0, v(\Sigma)) \to \R$ given by
\[ \cI(v) = \inf\{a(\Omega) \mid v(\Omega) = v\}. \]
By \cite[\S 3.10]{Morgan:reg-iso}, for all $v \in (0,v(\Sigma))$ there is an open set $\Omega\subset \Sigma$ with smooth boundary (not necessarily unique) achieving $\cI(v)$. In fact, $\Omega$ minimizes $a(\cdot)$ among all Caccioppoli sets with weighted volume $v$.

\subsection{First variation}
We compute the first variation of the functionals $a$ and $v$. Let $\{\Omega_t\}_{|t|<\eps}$ be a smooth family of open sets with smooth boundary whose variation vector field along $\partial\Omega = \partial\Omega_0$ is $f\nu$, where $\nu$ is the unit normal field to $\partial\Omega$ pointing out of $\Omega$.
The following computation is standard (see for instance \cite[\S3.2]{CL:soapbubbles}).

\begin{proposition}\label{prop:first_variation_iso}We have
    \[ \frac{d}{dt}\Big|_{t=0} a(\Omega_t) = \int_{\partial\Omega} (H + \alpha \theta^{-1}\langle \nabla^{\Sigma}\theta, \nu\rangle)f\theta^{\alpha}\ \ \text{and}\ \ \frac{d}{dt}\Big|_{t=0} v(\Omega_t) = \int_{\partial \Omega} f\theta^{\alpha}. \]
\end{proposition}

\subsection{Second variation}
We compute the second variation of the functionals $a$ and $v$. We consider the same setup as in the previous subsection. The computations follow similar arguments as in \cite[\S3.2]{CL:soapbubbles}.

\begin{proposition}\label{prop:second_variation_iso}We have
    \begin{align*}
        \frac{d^2}{dt^2}\Big|_{t=0} a(\Omega_t)
        & = \int_{\partial \Omega} |\nabla^{\partial \Omega}f|^2\theta^{\alpha} - (\Ric_{\Sigma}(\nu,\nu) + |A_{\partial \Omega}|^2)f^2\theta^{\alpha} + \alpha (\Delta^{\Sigma}\theta - \Delta^{\partial \Omega}\theta)f^2\theta^{\alpha-1}\\
        & \hspace{0.5cm} + \int_{\partial \Omega} \alpha(\alpha-1)\langle \nabla^{\Sigma}\theta, \nu\rangle^2f^2\theta^{\alpha-2} + H(H+\alpha \theta^{-1}\langle \nabla^{\Sigma} \theta, \nu\rangle)f^2\theta^{\alpha},
    \end{align*}
    and
    \[ \frac{d^2}{dt^2}\Big|_{t=0} v(\Omega_t) = \int_{\partial \Omega} (H + \alpha\theta^{-1}\langle \nabla^{\Sigma}\theta, \nu\rangle)f^2\theta^{\alpha}. \]
\end{proposition}
\begin{proof}
    By Propostion \ref{prop:first_variation_iso}, we compute
    \begin{align*}
        \frac{d^2}{dt^2}\Big|_{t=0} a(\Omega_t)
        & = \int_{\partial \Omega} (-\Delta^{\partial \Omega} f - (\Ric_{\Sigma}(\nu,\nu) + |A_{\partial \Omega}|^2)f)f\theta^{\alpha}\\
        & \hspace{0.5cm} + \int_{\partial \Omega} \alpha \Hess^{\Sigma} \theta(\nu, \nu)f^2\theta^{\alpha-1} - \alpha \langle \nabla^{\partial \Omega}\theta, \nabla^{\partial \Omega}f\rangle f\theta^{\alpha-1}\\
        & \hspace{0.5cm} + \int_{\partial \Omega} \alpha H \langle \nabla^{\Sigma} \theta, \nu\rangle f^2 \theta^{\alpha-1} + \alpha(\alpha-1)\langle \nabla^{\Sigma}\theta, \nu\rangle^2f^2\theta^{\alpha-2}\\
        & \hspace{0.5cm} + \int_{\partial \Omega} H(H+\alpha \theta^{-1}\langle \nabla^{\Sigma} \theta, \nu\rangle)f^2\theta^{\alpha}.
    \end{align*}
    Using the formula
    \[ \Hess^{\Sigma} \phi(\nu, \nu) = \Delta^{\Sigma}\phi - \Delta^{\partial \Omega} \phi - H\langle\nabla^{\Sigma}\phi, \nu\rangle \]
    and applying integration by parts to the $-f\theta^{\alpha}\Delta^{\partial \Omega}f$ term, we deduce
    \begin{align*}
        \frac{d^2}{dt^2}\Big|_{t=0} a(\Omega_t)
        & = \int_{\partial \Omega} |\nabla^{\partial \Omega}f|^2\theta^{\alpha} - (\Ric_{\Sigma}(\nu,\nu) + |A_{\partial \Omega}|^2)f^2\theta^{\alpha} + \alpha (\Delta^{\Sigma}\theta - \Delta^{\partial \Omega}\theta)f^2\theta^{\alpha-1}\\
        & \hspace{0.5cm} + \int_{\partial \Omega} \alpha(\alpha-1)\langle \nabla^{\Sigma}\theta, \nu\rangle^2f^2\theta^{\alpha-2} + H(H+\alpha \theta^{-1}\langle \nabla^{\Sigma} \theta, \nu\rangle)f^2\theta^{\alpha}.
    \end{align*}
    Similarly, by Proposition \ref{prop:first_variation_iso}, we compute
    \begin{align*}
        \frac{d^2}{dt^2}\Big|_{t=0} v(\Omega_t) = \int_{\partial \Omega} (H + \alpha\theta^{-1}\langle \nabla^{\Sigma}\theta, \nu\rangle)f^2\theta^{\alpha}.
    \end{align*}
    This completes the proof.
\end{proof}

\subsection{Differential inequality in the barrier sense}\label{subsec:diff_ineq}
Fix $v_0 \in (0, v(\Sigma))$.

Let $\Omega$ be a weighted isoperimetric set for the problem $\cI(v_0)$. Let $\{\Omega_t\}_{|t|<\eps}$ be a smooth family of open sets with smooth boundary with $\Omega_0 = \Omega$ whose variation vector field at $t=0$ is $\theta^{-\alpha}\nu$, where $\nu$ is the outward pointing unit normal vector field along $\partial \Omega$.

We note that $v(t) := v(\Omega_t)$ is a smooth function. By Proposition \ref{prop:first_variation_iso} with $f = \theta^{-\alpha}$, we have
\[ v'(0) = \frac{d}{dt}\Big|_{t=0} v(\Omega_t) = \int_{\partial \Omega} 1 > 0. \]
By the inverse function theorem, there is some small $\sigma > 0$ and a smooth function
\[ t : (v_0-\sigma, v_0+\sigma) \to \R \]
that is the inverse of $v(t)$.

Let $u : (v_0 - \sigma, v_0 + \sigma) \to \R$ be defined by
\begin{equation}\label{eqn:u_def}
    u(v) := a(t(v)).
\end{equation}
Note that $u(v_0) = a(0) = \cI(v_0)$. Moreover, since $v(\Omega_{t(v)}) = v$, we have $u(v) \geq \cI(v)$ for all $v \in (v_0 - \sigma, v_0 + \sigma)$.

Let primes denote derivatives with respect to $v$ and dots denote derivatives with respect to $t$.

\begin{proposition}\label{prop:diff_ineq1}
    The function $u$ satisfies
    \[ u''(v_0) \leq -\left(2+\frac{1}{2}u'(v_0)^2\right)u(v_0)^{-1}. \]
\end{proposition}
\begin{proof}
    We have
    \[ t'(v) = \frac{1}{\dot{v}(t(v))}\ \text{and}\ t''(v) = -\frac{\ddot{v}(t(v))}{\dot{v}(t(v))^3}. \]
    By Proposition \ref{prop:first_variation_iso} and \ref{prop:second_variation_iso}, we have
    \[ t'(v_0) = \left(\int_{\partial \Omega} 1\right)^{-1}\ \text{and} \quad  t''(v_0) = -\left(\int_{\partial \Omega} 1\right)^{-3}\int_{\partial \Omega}(H+\alpha\theta^{-1}\langle \nabla^{\Sigma}\theta, \nu\rangle)\theta^{-\alpha}. \]

    We note that
    \[ u'(v) = \frac{d}{dv}a(t(v)) = \dot{a}(t(v))t'(v) \]
    and
    \[ u''(v) = \frac{d^2}{dv^2}a(t(v)) = \ddot{a}(t(v))t'(v)^2 + \dot{a}(t(v))t''(v). \]

    We therefore have (by Proposition \ref{prop:first_variation_iso} with $f = \theta^{-\alpha}$)
    \[ u'(v_0) = H+\alpha\theta^{-1}\langle \nabla^{\Sigma}\theta, \nu\rangle, \]
    where we use the fact that isoperimetric surfaces satisfy the equation $H+\alpha\theta^{-1}\langle \nabla^{\Sigma}\theta, \nu\rangle = \lambda$ for some $\lambda \in \R$.

    By \eqref{eqn:eigenfunction} and Propositions \ref{prop:first_variation_iso} and \ref{prop:second_variation_iso} with $f = \theta^{-\alpha}$, we have
    \begin{align*}
        u''(v_0)\left(\int_{\partial \Omega} 1\right)^2
        & = \left(\ddot{a}(0)(t'(v_0))^2 + \dot{a}(0)t''(v_0)\right)\left(\int_{\partial \Omega} 1\right)^2\\
        & = \int_{\partial \Omega} |\nabla^{\partial \Omega}(\theta^{-\alpha})|^2\theta^{\alpha} - (\Ric_{\Sigma}(\nu,\nu) + |A_{\partial \Omega}|^2)\theta^{-\alpha}\\
        & \hspace{0.5cm} + \int_{\partial \Omega} \alpha (\Delta^{\Sigma}\theta - \Delta^{\partial \Omega}\theta)\theta^{-\alpha-1} + \alpha(\alpha-1)\langle \nabla^{\Sigma}\theta, \nu\rangle^2\theta^{-\alpha-2}\\
        & \hspace{0.5cm} + \int_{\partial \Omega} H(H+\alpha \theta^{-1}\langle \nabla^{\Sigma} \theta, \nu\rangle)\theta^{-\alpha} - (H+\alpha\theta^{-1}\langle \nabla^{\Sigma}\theta, \nu\rangle)^2\theta^{-\alpha}\\
        & = \int_{\partial \Omega} (\alpha\theta^{-1}\Delta^{\Sigma} \theta - \Ric_{\Sigma}(\nu,\nu))\theta^{-\alpha} -\alpha |\nabla^{\partial \Omega}\theta|^2\theta^{-\alpha-2} -  |A_{\partial \Omega}|^2\theta^{-\alpha}\\
        & \hspace{0.5cm} + \int_{\partial \Omega} -\alpha H \langle \nabla^{\Sigma}\theta, \nu\rangle \theta^{-\alpha-1} -  \alpha \langle \nabla^{\Sigma} \theta, \nu\rangle^2\theta^{-\alpha-2}\\
        & \leq \int_{\partial \Omega} -2\theta^{-\alpha} - \frac{1}{2}H^2\theta^{-\alpha} - \alpha H \langle \nabla^{\Sigma}\theta, \nu\rangle \theta^{-\alpha-1} - \alpha \langle \nabla^{\Sigma} \theta, \nu\rangle^2\theta^{-\alpha-2}\\
        & = \int_{\partial \Omega}-\left(2 + \frac{1}{2}(H+\alpha\theta^{-1}\langle \nabla^{\Sigma}\theta, \nu\rangle)^2\right) \theta^{-\alpha} + \frac{1}{2}\alpha(\alpha - 2) \langle \nabla^{\Sigma}\theta, \nu\rangle^2\theta^{-\alpha-2}\\
        & \leq -\left(2 + \frac{1}{2}u'(v_0)^2\right)\int_{\partial \Omega} \theta^{-\alpha},
    \end{align*}
    where we used $0 < \alpha \leq 2$. In particular, $u''(v_0)\leq 0$. 
    
    By H{\"o}lder's inequality, we have
    \[ \left(\int_{\partial \Omega} 1\right)^2 \leq \int_{\partial \Omega} \theta^{\alpha}\int_{\partial \Omega} \theta^{-\alpha} = u(v_0)\int_{\partial \Omega}\theta^{-\alpha}. \]
    Hence, $u$ satisfies
    \[ u''(v_0) \leq -\left(2 + \frac{1}{2}u'(v_0)^2\right)u(v_0)^{-1}, \]
    which concludes the proof.
\end{proof}

We consider a power of $\cI$ and $u$ to simplify the corresponding differential inequality. We let $\cF(v) = \cI(v)^{3/2}$. By Proposition \ref{prop:diff_ineq1}, we have the following result.

\begin{proposition}\label{prop:diff_ineq2}
    For any $v_0 \in (0, V)$, there is a smooth function $U : (v_0 - \sigma, v_0 + \sigma) \to \R$ satisfying
    \begin{itemize}
        \item $U(v_0) = \cF(v_0)$,
        \item $U(v) \geq \cF(v)$ for all $v \in (v_0 - \sigma, v_0 + \sigma)$, and
        \item $U''(v_0) \leq -3U(v_0)^{-1/3}$.
    \end{itemize}
\end{proposition}
\begin{proof}
    We take $U(v) = u(v)^{3/2}$ with $u$ defined as in \eqref{eqn:u_def}. By the definition of $u$ and $\cI$, the first two bullet points follows immediately. 
    
    We directly compute
    \[ U'(v) = \frac{3}{2}u^{1/2}(v)u'(v) \]
    and
    \begin{align*}
        U''(v_0)
        & = \frac{3}{4}u^{-1/2}(v_0)u'(v_0)^2 + \frac{3}{2}u(v_0)^{1/2}u''(v_0)\\
        & \leq \frac{3}{4}u^{-1/2}(v_0)u'(v_0)^2 - 3u(v_0)^{-1/2} - \frac{3}{4}u(v_0)^{-1/2}(u'(v_0))^2 = -3U(v_0)^{-1/3},
    \end{align*}
    where the inequality follows from Proposition \ref{prop:diff_ineq1}.
\end{proof}

From the existence of the upper barrier proved above, we can now conclude that $\cI$ is continuous.

\begin{proposition}\label{prop:cts-iso}
    $\cI$ is continuous.
\end{proposition}
\begin{proof}
    By the compactness theory for Caccioppoli sets and the lower semi-continuity of mass, we have
    \[ \liminf_{v \to v_0} \cI(v) \geq \cI(v_0). \]
    By the existence of a continuous upper barrier function for $\cI$ at $v_0$ for any $v_0 \in (0, v(\Sigma))$ (namely, take $U^{2/3}$ from Proposition \ref{prop:diff_ineq2}), we also have
    \[ \limsup_{v \to v_0} \cI(v) \leq \cI(v_0), \]
    so $\cI$ is continuous.
\end{proof}

\subsection{Solutions to the ordinary differential equation}
We study solutions to the ODE
\begin{equation}\label{eqn:ode}
    f''(v) = -3f(v)^{-1/3}.
\end{equation}

Observe that $g: [0,1)\to [0,\frac \pi 4)$ given by \[g(x) =  \frac 13 \int_0^x \frac{dt}{\sqrt{1-t^{2/3}}}\] is a diffeomorphism. As such, the map $v\mapsto g^{-1}(\frac \pi 4 -v)$ extends to an even smooth map $\tilde f : [-\frac \pi4,\frac\pi4]\to\R$ so that $\tilde f(0) = 1$ and $\tilde f''(v) = -3\tilde f(v)^{-1/3}$. For $z>0$ define $f_z : (-\frac \pi 4 z,\frac\pi4z) \to \R$ by 
\begin{equation}\label{eqn:fz}
v\mapsto z^{\frac 32} \tilde f(z^{-1}v).
\end{equation}
Note that $f_z(v)$ solves \eqref{eqn:ode} and has $f_z'(0) = 0$ and $f_z(0) = z^{3/2}$. We set $\beta(z) := \frac  \pi 4 z$ and observe that $\lim_{v \to \pm \beta(z)} f_z(v) = 0$. We extend $f_z$ to all of $\R$ by zero.

\subsection{Proof of the volume bound}

Our aim is to show that $v(\Sigma) = \int_{\Sigma} \theta^{\alpha} \leq 2\pi^2$. We assume for the sake of contradiction that $v(\Sigma) = \int_{\Sigma}\theta^{\alpha} > 2\pi^2$.

\begin{claim}
There is a $\delta > 0$ so that for $z = 4\pi + \delta$, we have 
\begin{equation}\label{eq:barrier-iso-profile}
\cF(v) \geq f_z(v - \beta(z))
\end{equation}
 for all $v \in (0, 2\pi^2)$.
\end{claim}
\begin{proof}
The key ingredient in the proof is the following observation in ODE comparison. For $c \in (0,1)$, set $f_{z,c}(v) = c^{3/4}f_z(v)$, where $f_z$ is defined in \eqref{eqn:fz}. It follows that $f_{z,c}$ solves the equation
\begin{equation}\label{eqn:odec}
    f'' = -3cf^{-1/3}.
\end{equation}
By Propositions \ref{prop:cts-iso} and \ref{prop:diff_ineq2} combined with the fact that $\cF$ is positive in $(0, v(\Sigma))$, it follows that no solution $f(v)$ to \eqref{eqn:odec} for $c \in (0,1)$ can touch $\cF(v)$ from below (meaning there is some $v_0$ so that $f(v) \leq \cF(v)$ for all $v \in (v_0 - \eps, v_0 + \eps)$ and $f(v_0) = \cF(v_0)$). Indeed, take $U$ to be the smooth function defined near $v_0$ from Proposition \ref{prop:diff_ineq2}. Then the smooth function $f$ touches the smooth function $U$ from below at $v_0$ and
\[ U''(v_0) \leq -3U(v_0)^{-1/3} = -3\cF(v_0)^{-1/3} < -3c\cF(v_0)^{-1/3} = -3cf(v_0)^{-1/3} = f''(v_0), \]
a contradiction.

Now let $\delta > 0$ and $\eps > 0$ be sufficiently small so that $\frac \pi 2 z + \eps z < v(\Sigma)$ for $z \in (0, 4\pi + \delta)$, which is possible since $v(\Sigma) > 2\pi^2 = 2\beta(4\pi)$. Consider the graph of 
\[ g_{z,c}(v) = f_{z,c}(v - \beta(z) - \eps z) \]
for $v \in [\eps z, 2\beta(z)+ \eps z]$. Note that
\[ g_{z,c}(\eps z) = g_{z,c}(2\beta(z) + \eps z) = 0 < \min\{\cF(\eps z), \cF(2\beta(z) + \eps z)\}. \]
Moreover, $g_{z,c}$ converges uniformly to zero as $z \to 0$. Hence, if $g_{z^*,c}(v^*) > \cF(v^*)$ for some $v^*$ and $z^*$, then there must be some $z \in (0, z^*]$ so that $g_{z,c}$ touches $\cF$ from below, which contradicts the ODE comparison observation above. Therefore, we have $\cF \geq g_{z,c}$ for every $z \in (0, 4\pi + \delta)$. We send $z$ to $4\pi + \delta$ and then $\eps$ to $0$, which proves the assertion for $f_{z,c}$. Then we send $c$ to $1$, which proves the claim. \end{proof}

We now study the asymptotic behavior of $\cF$ and $f_{4\pi+\delta}(v - \beta(4\pi+\delta))$ as $v \to 0$.

Since $f_{4\pi + \delta}(-\beta(4\pi+\delta)) = 0$, we have $f_{4\pi + \delta}'(-\beta(4\pi+\delta)) = 3\sqrt{4\pi + \delta}$. Hence, we have
\begin{equation}\label{eqn:asym_f}
    f_{4\pi+\delta}(v - \beta(4\pi+\delta)) = (3\sqrt{4\pi + \delta})v + o(v)
\end{equation}
as $v \to 0$.

On the other hand, we compute an upper bound for $\cF$ as $v \to 0$ by studying small geodesic balls. Take $x_0$ so that $\theta(x_0) = \min \theta = 1$. A straightforward computation gives
\[ v(B_r(x_0)) = \frac{4}{3}\pi r^3 + o(r^3) \]
and
\[ a(B_r(x_0)) = 4\pi r^2 + o(r^2) \]
as $r \to 0$. Solving the first equation for $r$ and plugging into the second equation, we deduce
\[ \cI(v) \leq (36\pi)^{1/3}v^{2/3} + o(v^{2/3}), \]
so
\begin{equation}\label{eqn:asym_F}
    \cF(v) \leq (6\sqrt{\pi})v + o(v).
\end{equation}

However, \eqref{eq:barrier-iso-profile}, \eqref{eqn:asym_f}, and \eqref{eqn:asym_F} imply $3\sqrt{4\pi + \delta} \leq 6\sqrt{\pi}$, which yields a contradiction. Therefore, since we normalized so that $\min\theta = 1$, we have
\begin{equation} \label{eq:concl-thmvol}\mathrm{Vol}(\Sigma, \gamma) \leq \int_{\Sigma} \theta^{\alpha} = v(\Sigma) \leq 2\pi^2, \end{equation}
which concludes the proof of Theorem \ref{thm:diameter_and_vol}.

\begin{remark}\label{rema:rigid-vol}
    We remark that the rigidity case of Theorem \ref{thm:diameter_and_vol} follows easily. We proved above that $v(\Sigma)\leq 2\pi^2$. As such, if $\mathrm{Vol}(\Sigma,\gamma) = 2\pi^2$ then \eqref{eq:concl-thmvol} and the normalization of $\theta$ gives that $\theta \equiv 1$ on $\Sigma$. Thus, \eqref{eqn:eigenfunction} gives $\lambda_{\Ric}(\gamma) \geq 2$. As such, the rigidity case in the classical Bishop--Gromov theorem implies that $(\Sigma,\gamma)$ is isometric to the round $3$-sphere. 
\end{remark}

\section{Stable Bernstein Theorem}\label{sec:proofs}
We use the estimates of the previous sections to prove the stable Bernstein theorem in $\R^5$ as well as the various consequences stated in the introduction.

Let $F: M^4 \to \R^5$ be a complete, simply connected, two-sided stable minimal immersion of codimension one, and let $g$ denote the pullback metric on $M$. On $N = M \setminus F^{-1}(\{0\})$, let $\tilde{g} = r^{-2}g$, where $r$ is the Euclidean distance function from ${0}$.

Combining Theorem \ref{thm:spectralbiRic}, Theorem \ref{thm:mububbles}, and Theorem \ref{thm:diameter_and_vol} (with the appropriate rescaling of the $\mu$-bubble metric), we have the following tool.

\begin{lemma}\label{lem:tool}
    Let $X \subset N$ be a closed subset with boundary $\partial X = \partial_+X \sqcup \partial_-X$. Suppose $d_{\tilde{g}}(\partial_+X, \partial_-X) \geq 10\pi$. Then there is a connected, relatively open subset $\Omega \subset X$ with smooth boundary $\partial \Omega = \partial_-X \sqcup \Sigma$ so that
    \begin{itemize}
        \item $\partial_-X \subset \Omega$,
        \item $\Sigma \subset X \setminus \partial X$ is a closed hypersurface,
        \item $\Omega \subset B_{10\pi}(\partial_-X)$, and
        \item any connected component $\Sigma_0$ of $\Sigma$ has intrinsic diameter at most $2\pi$ and volume at most $16\pi^2$.
    \end{itemize}
\end{lemma}

We first prove \emph{a priori} that any complete, simply connected, two-sided stable minimal immersion $M^4\to\R^5$ has Euclidean volume growth. The proof follows the same strategy as \cite{CL:aniso}. 

\begin{theorem}\label{thm:Euc_vol_growth}
    Any complete, simply connected, two-sided stable minimal immersion $M^4 \to \R^5$ satisfies
    \[ \cH^4(B_{\rho}(x_0) \subset M) \leq 8\pi^2e^{44\pi}\rho^4\]
for all $\rho>0$ and $x_0\in M$.
\end{theorem}
\begin{proof}
Without loss of generality we can assume that $0 \in F(M)$ and that $F(x_0) = 0$. Given $\rho > 0$, by Lemma \ref{lem:tool} (with $X = N \setminus B_{\rho}(x_0)$), there is a relatively open subset $\tilde{\Omega} \subset N \setminus B_{\rho}(x_0)$ such that
    \begin{itemize}
        \item $\partial B_{\rho}(x_0) \subset \tilde{\Omega}$,
        \item $\tilde{\Omega} \subset \tilde{B}_{10\pi}(\partial B_{\rho}(x_0))$, and
        \item any connected component of $\partial \tilde{\Omega}\setminus \partial B_{\rho}(x_0)$ is smooth with $\tilde{g}$-volume at most $16\pi^2$.
    \end{itemize}
    Converting back to the metric $g$ (see \cite[Lemma 6.2]{CL:aniso}), we have
    \[ \tilde{\Omega} \subset \{x \in M : d_M(x,\partial B_{\rho}(x_0)) < \rho e^{10\pi}\} \subset B_{\rho e^{11\pi}}(x_0). \]
    Hence, any connected component of $\partial \tilde{\Omega}\setminus \partial B_{\rho}(x_0)$ has $g$-volume at most $16\pi^2e^{33\pi}\rho^3$. Since $M$ is simply connected and has one end (by \cite{CSZ:end}), there is a precompact open set $\Omega$ containing $B_{\rho}(x_0) \cup \tilde{\Omega}$ that has exactly one boundary component, which is one of the components of $\partial \tilde{\Omega}\setminus \partial B_{\rho}(x_0)$. By the isoperimetric inequality for minimal hypersurfaces in $\R^{n+1}$ \cite{MS:sobolev} (we use the sharp version due to Brendle \cite{Brendle:iso} to improve the resulting constant), we have
    \[ \cH^4(B_\rho(x_0) \subset M) \leq c_{\text{iso}}(16\pi^2)^{4/3}e^{44\pi}\rho^4, \]
    where $c_{\text{iso}} = (128\pi^2)^{-1/3}$ in this dimension. This concludes the proof.
\end{proof}

\begin{proof}[Proof of Theorem \ref{thm:main}]
Suppose that $M^4\to\R^5$ is a complete, two-sided stable minimal immersion. Since stability passes to the universal cover, we can assume that $M$ is simply connected. Theorem \ref{thm:Euc_vol_growth} implies that $M$ has (intrinsic) Euclidean volume growth. Thus, \cite{SSY}  implies that $M$ is flat. 
\end{proof}

\begin{proof}[Proof of Corollary \ref{coro:curv-est}]
Arguing exactly as in \cite[Corollary 2.5]{CLS:stable}, if the asserted curvature estimates failed, then there would exist a non-flat complete, two-sided stable minimal immersion $M^4\to\R^5$. This contradicts Theorem \ref{thm:main}. 
\end{proof}

\begin{proof}[Proof of Theorem \ref{thm:localize}]
    By Lemma \ref{lem:tool}, there is a relatively open set $\tilde{\Omega} \subset M\setminus F^{-1}(\{0\})$ so that
    \begin{itemize}
        \item $\partial M \subset \tilde{\Omega}$,
        \item $\tilde{\Omega} \subset \tilde{B}_{10\pi}(\partial M)$, and
        \item any connected component of $\partial \tilde{\Omega} \setminus \partial M$ is smooth with $\tilde{g}$-volume at most $16\pi^2$.
    \end{itemize}
    Since $M$ is simply connected and $\partial M$ is connected, we can find an open subset $M' \Subset \mathrm{Int}(M)$ so that $\partial M'$ is connected and consists of one of the components of $\partial \tilde{\Omega} \setminus \partial M$. Converting between distances in $g$ and distances in $\tilde{g}$ by \cite[Lemma 25]{CL:aniso}, we have $M_{\rho_0}^* \subset M'$, where $\rho_0 = e^{-11\pi}$ and $M_{\rho_0}^*$ is the connected component of $F^{-1}(B_{\R^5}(0, \rho_0))$ that contains $x_0$. By the same calculation as in the proof of Theorem \ref{thm:Euc_vol_growth}, we have
    \[ \cH^4(M_{\rho_0}^*) \leq 8\pi^2.\]
    This finishes the proof.
\end{proof}

\begin{proof}[Proof of Theorem \ref{theo:morse}]
We follow \cite[\S 6]{CL:R4}. By \cite[\S 3]{T}, if a complete, two-sided minimal immersion $M^4\to\R^5$ has $\int_M |A_M|^4<\infty$, then $M$ has finite Morse index.

Now assume $M$ has finite Morse index. By \cite{li-wang-finite-index}, $M$ has finitely many ends and the space of $L^2$ harmonic 1-forms is finite. Recall by \cite[Proposition 1]{FC} that there is a compact subset $C \subset M$ so that $M \setminus C$ is stable. Fix $x_0 \in M$. Arguing as in \cite[Lemma 21]{CL:R4}, there is a $k<\infty$ so that for any $\rho>0$ sufficiently large (ensuring that $C \subset B_\rho(x_0))$, if $\Omega$ is a bounded open subset with smooth boundary satisfying
\begin{itemize}
    \item $B_\rho(x_0)\subset \Omega$, and
    \item $M\setminus\Omega$ has no bounded components,
\end{itemize}
then $\partial\Omega$ has $k$ components. The proof of Theorem \ref{thm:Euc_vol_growth} directly carries over to prove that
    \[ \cH^4(B_{\rho}(x_0) \subset M) \leq k\, 8\pi^2e^{44\pi}\rho^4\]
for all $\rho\gg0$ sufficiently large. Using the $L^4$-estimates from \cite{SSY} on the exterior region of $M$ we thus conclude that $\int_M |A_M|^4<\infty$.

It follows from \cite[Theorem A]{Anderson} (cf.\ \cite{SZ,T}) that the ends of $M$ are graphical over hyperplanes with bounded slope (see \cite[p.\ 22]{Anderson}). Granted this, \cite[Proposition 3]{S:sym} gives they are regular at infinity. This completes the proof. 
\end{proof}

\bibliographystyle{amsplain}
\bibliography{bib.bib}

\end{document}